\documentclass[a4paper,10pt]{article}
\usepackage{amsmath}
\usepackage{amssymb}

\usepackage{amscd}
\usepackage{amsthm}
\usepackage{url}
\usepackage{graphicx}

\usepackage{xcolor}

\newtheorem{thm}{Theorem}
\newtheorem{cor}[thm]{Corollary}

\newtheorem{lem}[thm]{Lemma}

\newtheorem{defin}[thm]{Definition}

\title{Plaque Inverse Limit of a Dynamical System - Dynamics, Signatures and Local Topology.}
\author{Avraham Goldstein\footnote{CUNY, USA. Email: avraham.goldstein.nyc@gmail.com.}}

\begin{document}
\maketitle

\begin{abstract}
The Plaque Inverse Limit of a branched covering self-map of a Riemann surface was introduced and studied in \cite{CCG}. A point $x$ of P.I.L. was called regular if P.I.L. has the natural Riemann Surface structure at $x$ and was called irregular otherwise. The notion of the signature $sign(x,c)$ of $x$ with respect to a critical point $c$, which was shown to be a local invariant of P.I.L. was introduced and developed. It was shown that $sign(x,c)$ is nontrivial for some critical points $c$ if and only if $x$ is an irregular point. It was shown that the local topology of P.I.L. at an irregular point $x$ has a property, that removing $x$ from any its neighborhood breaks some path-connected component of that neighborhood into an uncountable number of path-connected components. Finally, various signatures, including signatures of the invariant lifts of super-attracting and attracting cycles and certain signatures of the invariant lift of a parabolic cycle, were computed. All these signatures had a maximal element.
\\ \\
In this work we show that the local topology of P.I.L. at irregular points with different types of signatures is different. Namely, we prove that the local topology at an irregular point $x$ has a property, that for any neighborhood $V$ of $x$ and for some point $y\ne x$ in $V$, the open set $V-\{y\}$ consists of uncountable number of path-connected components, if and only if the signature $sign(x,c)$, for some critical point $c$, has no maximal element. Next, for a polynomial functions, we compute the signature of the invariant lift of a parabolic cycle with respect to a certain recurrent critical point. This signature, unlike the cases studied in \cite{CCG}, has no maximal element. We show that all other irregular points, except the invariant lifts of super-attracting, attracting, and parabolic cycles, have no maximal element with respect to some recurrent critical point.
\end{abstract}

{\bf Keywords:} Inverse limit, Riemann surface lamination, holomorphic dynamics, branched covering, local topology, irregular point, signature.

\section{Introduction.}
Inverse limits of iterations of branched self-coverings were introduced and studied in literature since the late 1920s. The most famous classical examples of such inverse limits are the $d$-adic solenoids, which are defined as the inverse limits of the iterates of the $d$-fold covering self-map $f(z)=z^d$ (where $d>1$) of the unit circle $S^1$. These inverse limits are compact, metrizable topological spaces that are connected,
but neither locally connected nor path connected. Solenoids were first introduced by L. Vietoris in 1927 for $d=2$ (see \cite{Vietoris}) and later in 1930 by van Dantzig for an arbitrary $d$ (see \cite{Danzig}).
\\ \\
In 1992 D. Sullivan (see \cite{Sullivan}) introduced Riemann surface laminations, which arise when taking inverse limits in dynamics. A Riemann surface lamination is locally the product of a complex disk and a Cantor set. In particular, D. Sullivan associates such lamination to any smooth, expanding self-maps of the circle $S^1$, with the maps $f(z)=z^d$ being examples of such maps.
\\ \\
In 1997 M. Lyubich and Y. Minsky (see \cite{LM}) and, in parallel, M. Su (see \cite{Su}) formalized the theory of Riemann surface laminations associated with dynamics of rational self-maps of the Riemann sphere. They start by considering the standard (Tychonoff) inverse limit of the iterations of a rational self-map of the Riemann sphere, which are regarded as just iterations of a continuous branched covering self-map of a Hausdorff topological space. Next, they introduce the notion of a regular point $-$ a point of the inverse limit is called regular if the pull-back of some open neighborhood of its first coordinate along that point eventually becomes univalent. Otherwise, a point is called irregular. They call the set of all regular points of the inverse limit, which is the inverse limit with all the irregular points removed, ``the regular set". The Riemann surface lamination, which they associate with a holomorphic dynamical system, in many cases, is just the regular set. In general, certain modifications are performed to the regular set, in order to satisfy the requirement, that the conformal structure on the leaves of the Riemann surface lamination is continuous along the fiber of the lamination. For the details of Lyubich-Minsky's definition and construction of the Riemann surface lamination, which are somewhat elaborate, we refer to \cite{LM}.
\\ \\
In 2014 C. Cabrera, C. Cherif and A. Goldstein (see \cite{CCG}) introduced and studied plaque inverse limits of the iterations of a branched covering self-map of a simply-connected Riemann surface. Plaque inverse limit is the inverse limit in the category of locally-connected Hausdorff topological spaces and continuous open maps. The open neighborhoods of a point in a plaque inverse limit, which constitute a local basis for its topology, are the pull-backs of the open neighborhoods of the first coordinate of this point along the point. The notions of regular and, by complement, irregular points for a plaque inverse limit are defined just like in the Lyubich-Minsky theory. Cabrera, Cherif and Goldstein show that for every irregular point $x$, there exists an open neighborhood $U$, such that for any open neighborhood $V\subset U$ of $x$, deleting $x$ from $V$ breaks some path-connected component of $V$ into an uncountable number of path-connected components. Thus, a plaque inverse limit does not have a manifold structure at the irregular points. Next, they introduce the notion of the signature $\sigma$-lattice, which is a $\sigma$-lattice of totally ordered sets of equivalence classes of binary sequences. With each point $x$ of the plaque inverse limit and each critical point $c$ of $f$ they associate and element $sign(x,c)$ in the signature $\sigma$-lattice, which is called the signature of $x$ with respect to $c$. Cabrera, Cherif and Goldstein proved that $sign(x,c)$ is a local invariant of the plaque inverse limit at $x$, which, for some $c$, becomes nontrivial if and only if $x$ is an irregular point. Next, they construct various irregular points and compute their signatures. They show, that the signatures of the irregular points, which are the invariant lifts of super-attracting and attracting cycles, with respect to every critical point, have maximal elements, while the signatures of all the points of the invariant lift of the boundary of certain Siegel disks, with respect to some critical points, have no maximal elements. They also consider infinitely renormalizable maps with \textit{a priori} bounds, including the case of quadratic map with the Feigenbaum parameter, construct certain irregular point, associated with these maps, and make some computations of its signature.
\\ \\
In this paper we:
\begin{itemize}
 \item Show that the local topology of plaque inverse limit at an irregular point, whose signatures, with respect to every critical point, have maximal elements, differs from the local topology at an irregular point, whose signature, with respect to some critical point, has no maximal element. Namely, we show that the signature of $x$, with respect to some critical point, has no maximal element if and only if there exists an open neighborhood $U$ of $x$, such that for any open neighborhood $V\subset U$ of $x$ there exists a point $v\in V$, different from $x$, such that $V-\{v\}$ consists of an uncountable number of path-connected components.
 \item Study various cases of invariant lifts of parabolic cycles. We show that in some of these cases the signatures of the irregular points, with respect to every critical point, have maximal elements. We show in case in which the signature, with respect to a certain recurrent critical point, has no maximal element. It is not currently known if such a critical point exists. We perform some explicit calculations of signatures for all these cases.
 \item Show all the irregular points, except the invariant lifts of super-attracting, attracting and parabolic cycles, have signatures with no maximal element with respect to some recurrent critical point.
\end{itemize}

\section{Definitions, Notations and Constructions.}
An inverse dynamical system is a sequence: $$S_1\;^{\underleftarrow{\;\;f_1\;\;}} S_2\;^{\underleftarrow{\;\;f_2\;\;}} S_3\;...$$ of Riemann surfaces $S_i$ and holomorphic branched coverings $f_i:S_{i+1}\rightarrow S_i$ where all $S_{i}$ are equal to a given Riemann surface $S_0$ and all $f_{i}$ are equal to a given holomorphic branched covering map $f:S_0\rightarrow S_0$ of degree $d$. In this work we assume that $1<d<\infty$ and $S_0$ is simply-connected $-$ either the unit disk, the complex plane or the Riemann sphere. The critical points of $f:S_0\rightarrow S_0$ are denoted by $c_1,...,c_{k}$. Abusing the notations, we, for all $i$, regard $f$ as a map from $S_{i+1}$ onto $S_{i}$ and regard $c_1,...,c_{k}$ as points of every $S_i$.
\\ \\
The Plaque Inverse Limit [P.I.L.] $S_{\infty}$ of an inverse dynamical system, introduced in \cite{CCG}, is the inverse limit in the category of locally connected topological spaces and continuous open maps. Thus, the underlying set of P.I.L. is the set of all the sequences $x=(x_1\in S_1,x_2\in S_2,...)$ of points, such that $f_{i}(x_{i+1})=x_i$ for $i=1,2,...$. The topology of P.I.L. is the set of all the sequences $U=(U_1\subset S_1,U_2\subset S_2,...)$ of open sets, such that $f_{i}(U_{i+1})=U_i$ for $i=1,2,...$. Finally, P.I.L. comes equipped with canonical projection maps $p_i:S_{\infty}\rightarrow S_i$, satisfying $p_i=f_i\circ p_{i+1}$ for all $i$, where $p_i$ takes $(x_1,x_2,...)\in S_{\infty}$ to $x_i\in S_i$. In this work we will be interested both in the plaque inverse limit $S_{\infty}$ and in its underlying topological space $T_{\infty}$, which comes without the projection maps onto $S_i$.
\\ \\
Recall, that the standard inverse limit $\bar{S}_{\infty}$ of the iterations of $f:S_0\rightarrow S_0$ $-$ the inverse limit inverse limit in the category of topological spaces and continuous maps $-$ has the same underlying set as the P.I.L., but is equipped with the Tychonoff topology. In the Tychonoff topology, the open sets are all sequences $U=(U_1\subset S_1,U_2\subset S_2,...)$ of open sets, where $f_{i}(U_{i+1})=U_i$ for $i=1,2,...$, such that there exists some number $t$, so that $f^{-1}_{i}(U_{i})=U_{i+1}$ for all $i>t$. Thus, P.I.L. has more open sets than the standard inverse limit. To be more precise, the open sets of P.I.L. are all the connected components of all the open sets of the standard inverse limit. The projections maps $p_i$ are the same for both inverse limits.
\\ \\
Recall, that a local basis for the topology of $S_{\infty}$ at a point $x$ consists of all open sets $U$, containing $x$, such that each $U_i$ is conformally equivalent to the unit disk in the complex plane. Each $f_i$, restricted to $U_{i+1}$, is conformally equivalent to some self-map $z^t$ of the unit disk of a degree $t$, between $1$ and $d$. Such open sets $U$ are called plaques. When we speak of a neighborhood of a point in $S_{\infty}$, we always assume it to be a plaque. Similarly, when we speak of a neighborhood of a point in a Riemann surface, we assume it to be simply connected.
\\ \\
Recall, that a point $x\in S_{\infty}$ is called regular if, for some neighborhood $U$ of $x$, there exists $n$, such that $U_{n+i+1}$ contains no critical points of $f_{n+i}$ for all $i=0,1,2,...$. Thus, $f_{n+i}:U_{n+i+1}\rightarrow U_{n+i}$ is a conformal equivalence. Otherwise, the point $x\in S_{\infty}$ is called irregular. Clearly, at a regular point P.I.L. has a natural Riemann Surface structure. It was shown in \cite{CCG} that if $x\in S_{\infty}$ is irregular then exists some open neighborhood $U$ of $x$, such that for any open neighborhood $V\subset U$ of $x$, removing $x$ from $V$ breaks some path-connected component of $V$ into an uncountable number of connected components. Thus, a point $x\in S_{\infty}$ is regular if and only if exists some neighborhood $U$ of $x$ which is topologically homeomorphic to an open disk.
\\ \\
In order to construct local invariants of P.I.L., called signatures, \cite{CCG} introduced the Boolean algebra $I$ of all classes of almost equal binary sequences and the $\sigma$-lattice $A$, spanned by sets $\alpha(a)\subset I$ for all $a\in I$, where $a\in I$, where $\alpha(a)$ is the set of all $b\in I$ such that $b\le a$. To each point $x$ of the P.I.L. and to each critical point $c$ of $f$, a unique element $sign(x,c)$ of $A$ was associated. It was called the signature of $x$ with respect to $c$. It was proved that $x$ is a regular point of P.I.L. if and only if the signature of $x$ with respect to all the critical points of $f$ is trivial.

\begin{defin} We denote the inverse system, associated with the iterations $f:S_0\rightarrow S_0$, by $S$, the plaque inverse limit of $S$ by $S_{\infty}$, and the underlying topological space of $S_{\infty}$ by $T_{\infty}$. \end{defin}

\begin{defin} An open set $U=(U_1,U_2,...)\subset S_{\infty}$ is called a plaque if each $U_i$ is conformally equivalent to the unit disk and $f$, restricted to $U_{i+1}$, is conformally equivalent to a self-map $z^t$ of the unit disk of a degree $t\leq d$.
\end{defin}
In this work, whenever we consider an open neighborhood of a point in a Riemann surface, we assume it to be simply connected. Similarly, whenever we consider an open neighborhood of a point in $S_{\infty}$, we assume it to be a plaque. Open neighborhoods of points in $T_{\infty}$, which we consider, are assumed to be connected.
\\ \\
Recall the following definitions and results from \cite{CCG}:
\\ \\
The set of binary sequences, equipped with the operations $\vee$, $\wedge$, $\neg$, defined by performing the binary operations \textbf{or}, \textbf{and}, \textbf{not}, respectively, in each coordinate of the sequences, is a Boolean algebra. Its partial order $\le$ is defined by $b\le a$ if and only if $a\vee b=a$. Its minimal and maximal elements are $(0,0,0,...)$ and $(1,1,1,...)$, respectively. Two binary sequences are called almost equal if they differ only in a finite number of places. This ``almost equality" is an equivalence relation, which respects the $\vee$, $\wedge$, $\neg$ operations, the partial order $\le$, and the minimal and maximal elements. Thus:

\begin{defin} The quotient  $I$ of the Boolean algebra of binary sequences by the ``almost equality" equivalence relation is the Boolean algebra of all classes of almost equal binary sequences, equipped with the $\vee$, $\wedge$, $\neg$ operations and the partial order $\le$. Its minimal element is $\textbf{0}=[0,0,...]$ and maximal element is $\textbf{1}=[1,1,...]$.\end{defin}

\begin{defin}
For every $a\in I$, the $\alpha(a)\subset I$ is defined as the set of all $b\in I$ such that $b\le a$.
\end{defin}
Note that $\alpha(a)\cup \alpha(b)\subset \alpha(a\vee b)$, $\alpha(a\wedge b)=\alpha(a)\cap \alpha(b)$, $\alpha(\textbf{0})=\{\textbf{0}\}$, and $\alpha(\textbf{1})=I$.

\begin{defin} \label{categ}
The $\sigma$-lattice $A$, spanned by all $\alpha(a)$, where $a\in I$, with the operations $\cup$ and $\cap$, the minimal
element $\{\textbf{0}\}$, and the maximal element $I$, is called the signature $\sigma$-lattice. The elements of $A$ are called signatures.
\end{defin}
Notice, that $\subset$ defines a partial order on $A$, which is consistent with the partial order $\leq$ of $I$ under the map
$\alpha$.

\begin{defin} The map $shift_m:I\rightarrow I$, for any integer $m$, takes each class $[i]\in I$ to the class of the binary sequence, obtained from $i$ by adjoining $m$ initial $0$ entries to it if $m\ge 0$ or by deleting $m$ initial entries from it if $m<0$. The map $shift_m:I\rightarrow I$ induced the map $shift_m:A\rightarrow A$
\end{defin}
Notice, that $shift_0=Id_I$ and $shift_m\circ shift_{-m}=Id_I$ for all $m$.
\\ \\
The following theorem and its corollary, which are Theorem 10 and Corollary 11 in \cite{CCG}, are crucial for defining signatures $sign(x,c)$ of points $x\in S_{\infty}$ and for distinguishing between signatures $sign(x,c)$ with and without maximal element:

\begin{thm}\label{thm.stab} Let $[i_1],[i_2],[i_3],...$ and $[t_1],[t_2],[t_3],...$ be elements of $I$. If $$\alpha[i_1]\cup\alpha[i_2]\cup\alpha[i_3]\cup ...=\beta=\alpha[t_1]\cap\alpha[t_2]\cap\alpha[t_3]\cap...$$ for some $\beta\in A$, then there exist some natural numbers $m$ and $n$ such that $$\alpha[i_1]\cup...\cup\alpha[i_m]=\beta=\alpha[t_1]\cap...\cap\alpha[t_n].$$ So,
$[i_1]\vee...\vee[i_m]=[t_1]\wedge...\wedge[t_n]$ and $\beta=\alpha([i_1]\vee...\vee[i_m])$.
\end{thm}

\begin{cor}\label{cor.finite} If $\alpha[i_1]\cap\alpha[i_2]\cap\alpha[i_3]\cap ...=\alpha[i]$ for some $[i]\in I$ then there exists a finite number $n$ such that $[i_1]\wedge...\wedge[i_n]=[i]$.
\end{cor}
Recall from \cite{CCG} the following Definitions and Lemmas, describing index, signature and shift operation:
\begin{defin} For an open neighborhood $U\subset S_{\infty}$ and a critical point $c\in S_0$, the index $ind(U,c)\in I$ of $U$ with respect to $c$ is the equivalence class of the binary sequence, which has 1 in its $n^{th}$ place if and only if $c\in U_n$.  \end{defin}
Notice, that if $V\subset U$, then $ind(V,c)\le ind(U,c)$.
\begin{defin} \label{oldwork.definsignature} For a point $x\in S_{\infty}$ and a critical point $c\in S_0$ the signature $sign(x,c)$ of $x$ with respect to $c$ is defined as $$sign(x,c)=\bigcap_{j=1}^{\infty}\alpha([ind(U(j),c)]),$$ where $(U(1),U(2),...)$ is an arbitrary sequence of open neighborhoods of $x$ in $S_{\infty}$ shrinking to $x$. \end{defin}
\begin{defin} For every integer $m$, the map $shift_m:I\rightarrow I$ takes each class $[i]\in I$ to the class of the binary sequence, obtained from $i$ by adjoining $m$ initial $0$ entries to it if $m\ge 0$ or by deleting $m$ initial entries from it if $m<0$.
\end{defin}
Clearly, $shift_0=Id_I$ and $shift_m\circ shift_{-m}=Id_I$ for all $m$.
\begin{lem} The maps $shift_m:I\rightarrow I$ induce maps $shift_m:A\rightarrow A$. Again, $shift_0=Id_A$ and $shift_m\circ shift_{-m}=Id_A$.
\end{lem}
\begin{lem} \label{oldwork.signatureinvariant} The signature $sign(x,c)$ does not depend on the choice of the sequence $(U(1),U(2),...)$. \end{lem}
\begin{lem} \label{oldwork.signatureregular} A point $x\in S_{\infty}$ is regular if and only if $sign(x,c)=\{[0,0,0,...]\}$ for every critical point $c$ of $f$. \end{lem}
\begin{lem}\label{lemma.signatures} For any $x,x'\in S_{\infty}$, if $sign(x,c)\cap sign(x',c)$ contains any element other than $[0,0,0,...]$, then $x=x'$. \end{lem}
\begin{lem} \label{oldwork.signatureperiod.lem} For any integer $m$  and any point $x\in S_{\infty}$, we have:
$$sign(f^m(x),c)=shift_{-m}(sign(x,c)).$$\end{lem}
The signature $sign(x,c)$ has a maximal element if exists some $a\in I$ such that $sign(x,c)=\alpha(a)$. Corollary \ref{cor.finite} implies, that $sign(x,c)$ has a maximal element if and only if exists some neighborhood $U\subset S_{\infty}$ of $x$ such that $sign(x,c)=ind(U,c)$.
\section{Signatures and local topology at irregular points}
The crucial technique, which permits us to characterize different local topologies at points with different types of signatures, is the interplay between local connectivity and local path-connectivity. This technique has already been introduced in utilized in \cite{CCG} and used to produce Theorem 15 there.  Here we further develop this technique. This permits us to make much sharper distinctions between various local topologies of points of P.I.L.
\\ \\
Two technical lemmas, which made this technique possible, were Lemmas 13 and 14 in \cite{CCG}. Due to their importance, we reproduce them, without proofs, here as Lemmas \ref{lemma.path} and \ref{lemma.neighborhood}. Next, we state Theorem 15 of \cite{CCG}, in a somewhat stronger form, as our Theorem \ref{irreg.topol.old}. Next, in Lemma \ref{Tychonoff} we make an important observation, which, together with technical Lemmas \ref{lemma.neighborhood.nomax.1}, \ref{lemma.neighborhood.nomax.2} and \ref{lemma.neighborhood.nomax.3}, permits us to enhance and refine the technique of juxtaposing local connectivity and local path-connectivity. Namely, we observe, that if a sequence of points of $S_{\infty}$ converges in Tychonoff topology, but does not converge in the plaque topology, then only finitely many of the points of that sequence can be contained in any compact set and, in particular, connected by a path in $S_{\infty}$. This enhanced and refined technique now distinguishes the local topology at the irregular points, which have a signature with no maximal element, from the local topology at the irregular points, which do not have such a signature.
\\ \\
Let $X$ be a regular, first-countable topological space and $z$ be a point in $X$.
\begin{defin} A sequence of open neighborhoods $(U(1),U(2),...)$ of $z$ shrinks to $z$ if $U(i+1)\subset U(i)$ for all $i$ and for any open neighborhood $V$ of $z$ there exists some $m$ such that $U(m)\subset V$. Thus, the set $\{U(1),U(2),...\}$ is a local base for the topology at $z$.
\end{defin}
\begin{defin} Given a sequence  $sq=(z(1),z(2),z(3),...)$ of points in $X$. A path $p:[0,1]\rightarrow X$ passes through the $sq$ in the correct way if there exist some numbers $0\leq t_1\leq t_2\leq ...\leq 1$ such that $z(m)=p(t_m)$ for all $m=1,2,...$.
\end{defin}
\begin{lem}\label{lemma.path} There exists a path $p:[0,1]\rightarrow X$ which passes through $sq$ in the correct way if and only if $sq$ converges to some point $z\in X$ and $$p(\lim_{m\rightarrow\infty}t_m)=z.$$
\end{lem}
\begin{lem}\label{lemma.neighborhood} For every irregular point $x\in S_{\infty}$, there exists some open neighborhood $U$ of $x$ such that for any open neighborhood $V\subset U$ of $x$, there are infinitely many positive integers $n(1)<n(2)<...$ for which $V_{n(i)}$ contains some critical points of $f_{n(i)-1}$ while $(U-V)_{n(i)}$ does not contain any critical points of $f_{n(i)-1}$.
\end{lem}
\begin{thm}\label{irreg.topol.old} The signature of $x\in S_{\infty}$ with respect to some critical point $c$ is nontrivial if and only if there exists an open neighborhood $U\subset T_{\infty}$ of $x$ such that for any neighborhood $W\subset U$ of $x$, deleting $x$ from $W$ breaks the path-connected component of $W$, containing $x$, into an uncountable number of path-connected components.
\end{thm}
\begin{proof}
The proof of Theorem 15 in \cite{CCG}, actually, establishes that for any irregular point $x$ there exists a neighborhood $U$ of $x$, such that deleting $x$ from any open neighborhood $W\subset U$ of $x$ breaks the path-connected component of $W$, containing $x$, into an uncountable number of path-connected components. To see this, notice, that the neighborhood $U$ in Lemma 14 of \cite{CCG} can always be taken smaller and renamed to $W$. But, due to Lemma \ref{oldwork.signatureregular}, $x$ is irregular if and only if the signature of $x$ with respect to some critical point $c$ is nontrivial.
\\ \\
For the other direction of the theorem, notice, that every regular point has a neighborhood, which is homeomorphic to an open unit disk.
\end{proof}
\begin{lem}\label{Tychonoff} Let $v=(v_1,v_2,...)$ be a point in $S_{\infty}$ and let $sq=(w(1),w(2),...)$ be a sequence of points in $S_{\infty}$ such that for some sequence of positive integers $(m(1)<m(2)<...)$, for every positive integer $n$ and for all $m\ge m(n)$ we have $w(m)_n=v_n$. In other words, for every $n$, all the entries of $sq$, after the initial $m(n)$ entries, have their first $n$ coordinates the same as $v$. Then, if $sq$ has a converging subsequence in $S_{\infty}$, the limit of that subsequence must be $v$.
\end{lem}
\begin{proof}
Let $v'=(v'_1,v'_2,...)$ be the limit of a subsequence $sq'=(w(j(1)),w(j(2)),...)$ of $sq$. For all $n$, the sequence $(w(j(1))_n,w(j(2))_n,...)$ of points in $S_n$ converges to $v'_n\in S_n$. But all $w(j(i))_n$, with $j(i)\ge m(n)$, are equal to $v_n$. Since $S_n$ is Hausdorff, this implies that $v'_n=v_n$.
\end{proof}
The following lemma is related to Lemma \label{lemma.neighborhood} and should be viewed as its extension:
\begin{lem}\label{lemma.neighborhood.nomax.1} For every point $x\in S_{\infty}$ such that for some critical point $c$, the signature $sign(x,c)$ has no maximal element, there exists an open neighborhood $U$ of $x$ in $S_{\infty}$, such that for any open neighborhood $V\subset U$ of $x$ there exists an open neighborhood $W$ of $x$, with its closure $\bar{W}$ contained inside $V$, and infinitely many positive integers $n(1)<n(2)<...$, so that each $(V-\bar{W})_{n(i)}$ contains the critical point $c$, while all the $(U-V)_{n(i)}$ do not contain any critical points of $f$.
\end{lem}
\begin{proof}
First, notice, that for any open neighborhood $V\subset S_{\infty}$ of $x$ we can always find some open neighborhood $W$ of $x$, with its closure
$\bar{W}$ contained in $V$, such that there are infinitely many positive integers $m(1)<m(2)<...$, for which the level $(V-\bar{W})_{m(i)}$ contains the critical point $c$. Indeed, if for all open neighborhoods $W$ of $x$, satisfying $\bar{W}\subset V$, the set $V-\bar{W}$ contains $c$ only in a finite number of its levels $(V-\bar{W})_n$, then for all open neighborhoods $W'\subset V$ of $x$, the set $(U-W')$ also contains $c$ only in a finite number of its levels $(U-W')_n$. That follows from the fact, that for any open neighborhood $W'$ of $x$ there exists an open neighborhood $W$ of $x$, whose closure $\bar{W}$ is contained inside $W'$. But, $(V-W')$ containing $c$ only in a finite number of its levels would imply that $ind(W',c)=ind(V,c)$. Since this would be true for all open neighborhoods $W'\subset V$ of $x$, we get that $sign(x,c)=\alpha(ind(V,c))$. But this would contradict our requirement that $sign(x,c)$ does not have a maximal element. So, for any neighborhood $V$ of $x$ there must exist infinitely many positive integers $m(1)<m(2)<...$, for which the level $(V-\bar{W})_{m(i)}$ contains the critical point $c$.
\\ \\
Next, if this lemma is false, then for any open neighborhood $U(1)$ of $x$ we can find some open neighborhood $V\subset U(1)$ of $x$, such that for any open neighborhood $W$ of $x$, whose closure $\bar{W}$ is contained inside $V$, and for any infinite sequence $sq=(n(1),n(2),...)$ of increasing positive integers, for which the levels $(V-\bar{W})_{n(i)}$ contain $c$, almost all the levels  $(U(1)-V)_{n(i)}$ of the set $U(1)-V$, except finitely many of them, contain some other critical points of $f$. Denote this $V$ by $U(2)$.
\\ \\
Now, for the open neighborhood $U(2)$ of $x$ we can find some open neighborhood $V\subset U(2)$ of $x$, such that for any open neighborhood $W$ of $x$, whose closure $\bar{W}$ is contained inside $V$, and for any infinite sequence $sq=(n(1),n(2),...)$ of increasing positive integers, for which the levels $(V-\bar{W})_{n(i)}$ contain $c$, almost all the levels $(U(2)-V)_{n(i)}$ of the set $U(2)-V$, except finitely many of them, contain some other critical points of $f$. Denote this $V$ by $U(2)$.
\\ \\
Proceed this way to define $U(3),U(4),...$.
\\ \\
So, for any positive integer $q$, the open neighborhoods $U(q)$ and $U(q+1)$ of $x$ have the property, that for any open neighborhood $W$ of $x$, whose closure $\bar{W}$ is contained inside $U(q+1)$, and for any infinite sequence $sq=(n(1),n(2),...)$ of increasing positive integers, for which the levels $(U(q+1)-\bar{W})_{n(i)}$ contain $c$, almost all the levels $(U(q)-U(q+1))_{n(i)}$ of the set $U(q)-U(q+1)$, except finitely many of them, contain some other critical points of $f$.
\\ \\
Now, fix some positive integer $q$, some open neighborhood $W$ of $x$ and some infinite sequence $sq=(n(1),n(2),...)$ of increasing positive integers, for which the set $U(q+1)-\bar{W}$ contains $c$ in all of its levels $(U(q+1)-\bar{W})_{n(i)}$. Clearly, for any integer $j$ between $1$ and $q$, the set $U(j+1)-\bar{W}$ also contains $c$ in all of its levels $(U(j+1)-\bar{W})_{n(i)}$, where $i=1,2,3,...$, while almost all the levels $(U(j)-U(j+1))_{n(i)}$ of the set $U(j)-U(j+1)$, except finitely many of them, contain some other critical points of $f$. Hence, for some positive integer $m$, all the open sets $(U(1)-U(2))_m$, $(U(2)-U(3))_m$, ..., $(U(q)-U(q+1))_m$ will contain some critical points of $f$. Since these sets are pairwise disjoint, all these critical points must be different.
\\ \\
But $f$ has only finitely many critical points, while we can fix $q$ as large as we want. This leads to a contradiction. Thus, the lemma cannot be false.
\end{proof}
\begin{lem}\label{lemma.neighborhood.nomax.2} The open neighborhood $W$ and the increasing integers $n(1),n(2),...$ in Lemma \ref{lemma.neighborhood.nomax.1} can be selected in such a way, that the sequence $(f^{n(1)-1}(c),f^{n(2)-1}(c),...)$ of points in $(V-\bar{W})_1=V_1-\bar{W}_{1}$ converges to some point $v_1\in V_1-\bar{W}_{1}$.
\end{lem}
\begin{proof}
For any open neighborhood $V\subset U$ of $x$ of Lemma \ref{lemma.neighborhood.nomax.1}, we can find an open neighborhood $V'$ of $x$, such that its closure $\bar{V}'$ is contained in $V$. By Lemma \ref{lemma.neighborhood.nomax.1}, there exists some open neighborhood $W'$ of $x$, with its closure $\bar{W}'$ contained inside $V'$, and infinitely many positive integers $n(1)<n(2)<...$, such that each $(V'-\bar{W}')_{n(i)}$ contains the critical point $c$, while all the $(U-V')_{n(i)}$ do not contain any critical points of $f$.
\\ \\
Since $(\bar{V}'-W')_1=\bar{V}'_1-W'_1$ is compact, we can find a subsequence $(n'(1),n'(2),...)$ of the sequence $(n(1),n(2),...)$, for which the sequence $(f^{n'(1)-1}(c),f^{n'(2)-1}(c),...$ converges to some point $v_1\in \bar{V}'_1-W'_{1}$. Now take $W$ to be any open neighborhood of $x$, such that its closure $\bar{W}$ is contained in $W'$. We get that $v_1\in V_1-\bar{W}_{1}$. Clearly, each $(V-\bar{W})_{n(i)}$ contains $c$, while all the $(U-V)_{n(i)}$ do not contain any critical points of $f$.
\end{proof}
\begin{lem}\label{lemma.neighborhood.nomax.3} The increasing integers $n(1),n(2),...$ in Lemma \ref{lemma.neighborhood.nomax.2} can be selected in such a way, that in each level $V_q$ of $V$, where $q=n(i),n(i)+1,...,n(i+1)-1$, we can find $2^i$ pairwise disjoint connected components $$\bar{W}_{q}(1),\bar{W}_{q}(2),\bar{W}_{q}(3),...,\bar{W}_{q}(2^i)$$ of the open set $f^{-q+1}(\bar{W}_1)$, so that for each $\bar{W}_{n(i+1)-1}(j)$, where $i=1,2,...$ and $j=1,2,...,2^i$, there are exactly two different sets $\bar{W}_{n(i+1)}(j_1)$ and $\bar{W}_{n(i+1)}(j_2)$ amongst the open sets $$\bar{W}_{n(i+1)}(1),\bar{W}_{n(i+1)}(2),\bar{W}_{n(i+1)}(3),...,\bar{W}_{n(i+1)}(2^{i+1}),$$ which map by $f$ onto $\bar{W}_{n(i+1)-1}(j)$.
\end{lem}
\begin{proof}
Let $m$ be the smallest integer, greater than $1$, such that $V_m-\bar{W}_m$ contains $c$. Then $V_m$ contains at least one more connected component $\bar{W}'_m$ of $f^{-1}(\bar{W}_{m-1})$, which is disjoint from $\bar{W}_m$. Indeed, since $V_{m-1}$ is a regular space, we can find an open, simply connected subset $\Omega_{m-1}$  of $V_{m-1}$, which contains $\bar{W}_{m-1}$ and does not contain $f(c)$. Let $\Omega_{m}$ be the connected component of $f^{-1}(\Omega_{m-1})$, which contains $\bar{W}_m$. Clearly,  $\Omega_{m}\subset V_{m}$. Since the simply connected sets $V_m$, $V_{m-1}$, $\Omega_{m}$ and $\Omega_{m-1}$ have Euler characteristic $1$, and since $f:V_{m}\rightarrow V_{m-1}$ and $f:\Omega_{m}\rightarrow \Omega_{m-1}$ are branched covering maps, it follows from the Riemann–Hurwitz formula, that the degree of $f:V_{m}\rightarrow V_{m-1}$ is greater than the degree of $f:\Omega_{m}\rightarrow \Omega_{m-1}$ by, at least, the degree (the ramification index) of $f$ at $c$ minus one. But the degree of $f$ at a critical point $c$ is greater than or equal to $2$. Thus, a generic point of $\bar{W}_{m-1}$ has a pre-image (under $f$) in $V_{m}$, which is not contained in $\bar{W}_{m}$. Hence, $V_m$ contains at least one more connected component $\bar{W}'_m$ of $f^{-1}(\bar{W}_{m-1})$, which is disjoint from $\bar{W}_m$.
\\ \\
Denote by $\bar{W}'_n$, for all $n\geq m$, the pre-image of $\bar{W}'_m$ under $f^{m-n}$ in $V_n$.
\\ \\
If the sequence $(n(1),n(2),...)$ of increasing positive integers from Lemma \ref{lemma.neighborhood.nomax.2} has an infinite subsequence $(n'(1),n'(2),...)$, for which $c$ is always contained in the set $V_{n'(i)}-\bar{W}'_{n'(i)}$, then just replace positive integers $n(1)<n(2)<...$ in Lemma \ref{lemma.neighborhood.nomax.2} by positive integers $n'(1)<n'(2)<...$. Notice, that for all $n'(i)$, where $i=1,2,...$, each connected component of $\bar{W}'_{n'(i+1)-1}$ has at least two connected components of $\bar{W}'_{n'(i+1)-1}$, which are mapped by $f$ onto $\bar{W}'_{n'(i+1)-1}$. Let $i(q)=0$ for $q<n'(1)$, and for $n'(1)\ge q$ let $i(q)$ be such a number, that $n'(i)\le q\le n'(i+1)-1$. We can, inductively on $q$, select some $2^{i(q)}$ pairwise disjoint connected components $\bar{W}_{q}(1),\bar{W}_{q}(2),\bar{W}_{q}(3),...,\bar{W}_{q}(2^{i(q)})$ of $f^{-q+1}(\bar{W}_1)$, in such a way, that for every $\bar{W}_{n'(i+1)-1}(j)$ there are two different $\bar{W}_{n'(i+1)}(j_1)$ and $\bar{W}_{n'(i+1)}(j_2)$ which are mapped onto $\bar{W}_{n'(i+1)-1}(j)$ by $f$.
\\ \\
If the sequence $(n(1),n(2),...)$ of increasing positive integers from Lemma \ref{lemma.neighborhood.nomax.2} does not have an infinite subsequence of indices $n$, for which $c$ is contained in $V_{n}-\bar{W}'_{n}$, then exists some integer $\upsilon$ such that for all $i=1,2,...$, the critical point $c$, which is contained in $V_{n(\upsilon+i)}$, is actually contained in its subset $\bar{W}'_{n(\upsilon+i)}$. But this implies, that for all $i=1,2,...$, the full pre-image $\hat{W}_{n(\upsilon+i)}$ in $V_{n(\upsilon+i)}$ of $\bar{W}_{n(\upsilon)}$ under $f^{n(\upsilon)-n(\upsilon+i)}$, does not contain the critical point $c$. Thus, just replace positive integers $n(1)<n(2)<...$ in Lemma \ref{lemma.neighborhood.nomax.2} by positive integers $n''(1)=n(\upsilon+1)<n''(2)=n(\upsilon+1)<...$. For all integers $n''(i)$, where $i=0,1,2,...$, each connected component of $\hat{W}'_{n''(i+1)-1}$ has at least two connected components of $\hat{W}'_{n''(i+1)-1}$, which are mapped by $f$ onto $\bar{W}'_{n''(i+1)-1}$. Let $i(q)=0$ for $q<n''(1)$, and for $n''(1)\ge q$ let $i(q)$ be such a number, that $n''(i)\le q\le n''(i+1)-1$. We can, inductively on $q$, select some $2^{i(q)}$ pairwise disjoint connected components $\bar{W}_{q}(1),\bar{W}_{q}(2),\bar{W}_{q}(3),...,\bar{W}_{q}(2^{i(q)})$ of $f^{-q+1}(\bar{W}_1)$, in such a way, that for every $\bar{W}_{n''(i+1)-1}(j)$ there are two different $\bar{W}_{n''(i+1)}(j_1)$ and $\bar{W}_{n''(i+1)}(j_2)$ which are mapped onto $\bar{W}_{n''(i+1)-1}(j)$ by $f$.
\end{proof}
\begin{thm} \label{nomax.main.1} Let $x\in S_{\infty}$ be an irregular point such that, for some critical point $c$ of $f:S_0\rightarrow S_0$, the signature $sign(x,c)$ has no maximal element. Then there exists some open neighborhood $U$ of $x$, such that for any open neighborhood $V$ of $x$, whose closure $\bar{V}$ is contained in $U$, there exists a point $v\in V$, different from $x$, such that the open set $V-\{v\}$ consists of an uncountable number of path-connected components.
\end{thm}
\begin{proof}
Let $U$ be an open neighborhood of $x$ such that, by Lemmas \ref{lemma.neighborhood.nomax.1}, \ref{lemma.neighborhood.nomax.2} and \ref{lemma.neighborhood.nomax.3}, for any open neighborhood $V\subset U$ of $x$, we can find an open neighborhood $W$ of $x$, whose closure $\bar{W}$ is contained inside $V$, and infinitely many positive integers $n(1),n(2),...$ so that:
\begin{itemize}
\item $(V-\bar{W})_{n(i)}$ contains the critical point $c$, while $(U-V)_{n(i)}$ does not contain any critical points of $f$;
\item The sequence $(f^{n(1)-1}(c),f^{n(2)-1}(c),...)$ converges to some point $v_1\in V_1-\bar{W}_{1}$;
\item For each $q=1,2,...$, we can select $2^i$ (where $i$ is $0$ for $q<n(1)$, and $i$ is such a number, that $n(i)\le q\le n(i+1)-1$, otherwise) pairwise disjoint connected components $\bar{W}_{q}(1),\bar{W}_{q}(2),\bar{W}_{q}(3),...,\bar{W}_{q}(2^i)$ of $f^{-q+1}(\bar{W}_1)$, in such a way, that for every $\bar{W}_{n(i+1)-1}(j)$ there are two different sets $\bar{W}_{n(i+1)}(j_1)$ and $\bar{W}_{n(i+1)}(j_2)$ among the sets $$\bar{W}_{n(i+1)}(1),\bar{W}_{n(i+1)}(2),\bar{W}_{n(i+1)}(3),...,\bar{W}_{n(i+1)}(2^{i+1}),$$ which map by $f$ onto $\bar{W}_{n(i+1)-1}(j)$.
\end{itemize}
Since each open set $(V-\bar{W})_{n(i)}$, for $i=1,2,...$, contains the point $c$, all the points $f^{n(q)-q}(c),f^{n(q+1)-q}(c),f^{n(q+2)-q}(c)...$ must belong to the open set $V_q-\bar{W}_{q}$ for all $q=1,2,3,...$. Since the point $v_1\in V_1-\bar{W}_{1}$ is the accumulation point of the sequence $(f^{n(1)-1}(c),f^{n(2)-1}(c),...)$, some pre-image $v_2\in V_2-\bar{W}_{2}$ of $v_1$ under $f$ must be an accumulation point of the set $\{f^{n(2)-2}(c),f^{n(3)-2}(c),...\}$. By the same logic, some pre-image $v_3\in V_3-\bar{W}_{3}$ of $v_2$ under $f$ must be an accumulation point of the set $\{f^{n(3)-3}(c),f^{n(4)-3}(c),...\}$. Continuing this way, we construct a point $v=(v_1,v_2,...)$ in $V-\bar{W}$, such that each $v_q$ is an accumulation point of the set $f^{n(q)-q}(c),f^{n(q+1)-q}(c),f^{n(q+2)-q}(c)...$ inside $V_q-\bar{W}_{q}$.
\\ \\
For any path $p:[0,1]\rightarrow V-\{v\}$, which connects any two points in $V-\{v\}$, there exists some positive integer $m$, such that the path $p_m:[0,1]\rightarrow V_m$ avoids the point $v_m\in V_m$. Otherwise, for all $i=1,2,...$, we can find points $v(i)\in V-\{v\}$ with $v(i)_i=v_i$, which belong to the compact subset $p([0,1])\subset V-\{v\}$ of $S_{\infty}$. By Lemma \ref{Tychonoff}, any converging subsequence of the sequence $(v(1),v(2),...)$ in $p([0,1])$ must converge to $v$, which is not contained in $V-\{v\}$. But this is a contradiction. Thus, for each $p:[0,1]\rightarrow V-\{v\}$ we have a positive integer $m$, such that the path $p_m:[0,1]\rightarrow V_m$ avoids the point $v_m\in V_m$.
\\ \\
Let $(\Theta(1),\Theta(2),...)$ be a sequence of open neighborhoods of $v$ in $V-\bar{W}$ shrinking to $v$. For every positive integer $m$, denote by $\Upsilon(m)$ the set of all paths $p:[0,1]\rightarrow V-\{v\}$, such that the path $p_m:[0,1]\rightarrow V_m$ does not enter the subset $\Theta(m)_m$ of $V_m$. In other words, $p_m([0,1])\cap \Theta(m)_m=\emptyset$. Clearly, for all $m$, $\Upsilon(m)$ is a subset of $\Upsilon(m+1)$. From the fact, that for any path $p:[0,1]\rightarrow V-\{v\}$ exists some positive integer $m$, such that the path $p_m:[0,1]\rightarrow V_m$ avoids the point $v_m\in V_m$, and from the fact, that all the sets $V_m$ are regular spaces, it follows that the union of all $\Upsilon(m)$, as $m=1,2,3,...$, is the entire set of all the paths $p:[0,1]\rightarrow V-\{v\}$.
\\ \\
For any $m=1,2,...$, consider any path $p:[0,1]\rightarrow V-\{v\}$ from the set of paths $\Upsilon(m)$, which connects the point $x$ to a point $x'\in V-\{v\}$, where $x'_1=x_1$ and each $x'_q$, for $q=2,3,...$, belongs to some set $W_q(j)$, described in the itemization above. Since, the path $p_m:[0,1]\rightarrow V_m$ avoids the subset $\Theta(m)_m$ of $V_m$, the closed path $p_1:[0,1]\rightarrow V-\{v\}$ has some looping number around almost all, except a finite number, of the points of the sequence $(f^{n(1)-1}(c),f^{n(2)-1}(c),...)$. Thus, the point $x$ can be connected to a countable number of different points $x'$, where $x'_1=x_1$ and each $x'_q$, for $q=2,3,...$, belongs to some $W_q(j)$, by paths from $\Upsilon(m)$. But the union of all $\Upsilon(m)$, as $m=1,2,3,...$, is the entire set of all the paths $p:[0,1]\rightarrow V-\{v\}$. But it follows from Lemma \ref{lemma.neighborhood.nomax.3} that there is an uncountable number of different points $x'$ in $V-\{v\}$, with $x'_1=x_1$ and with each $x'_q$, for $q=2,3,...$, belonging to some $W_q(j)$. Indeed, we have two different choices for each $x'_{n(i)}$, as $i=2,3,...$. Thus, $V-\{v\}$ has an uncountable number of path-connected components.
\end{proof}
\begin{cor} \label{nomax.main.cor.1} If $x\in S_{\infty}$ is an isolated irregular point, and for some critical point $c$ of $f:S_0\rightarrow S_0$, the signature $sign(x,c)$ has no maximal element, then there exists some open neighborhood $U$ of $x$, such that for any open neighborhood $V\subset U$ of $x$, the open set $V$ consists of an uncountable number of path-connected components.
\end{cor}
\begin{proof} By Theorem \ref{nomax.main.1}, there exists some open neighborhood $U'$ of $x$, such that for any open neighborhood $V\subset U$ of $x$ there exists a point $v\in V$, different from $x$, such that $V-\{v\}$ consists of an uncountable number of path-connected components. Since $x$ is an isolated irregular point, there exists some open neighborhood $W$ of $x$ in $S_{\infty}$, which does not contain any irregular points other than $x$. Take $U=U'\cap W$.
\\ \\
Now, for any open neighborhood $V\subset U$ of $x$ there exists a point $v\in V$, different from $x$, such that $V-\{v\}$ consists of an uncountable number of path-connected components. But this point $v$ is a regular point. Hence, there exists some open neighborhood $\Delta$ of $v$ in $S_{\infty}$, which is homeomorphic to an open unit disk. Thus, $\Delta-\{v\}$ is path-connected. So, removing $v$ from $V$ does not disconnect any path-connected components of $V$. Hence, $V$ must consist of an uncountable number of path-connected components.
\end{proof}
Theorem \ref{nomax.main.2} below asserts that the other direction in Theorem \ref{nomax.main.1} is also true:
\begin{thm} \label{nomax.main.2} Let $x\in S_{\infty}$ be an irregular point, such that for some every point $c$ of $f:S_0\rightarrow S_0$, the signature $sign(x,c)$ has a maximal element. Then there exists some open neighborhood $U$ of $x$, such that for any open neighborhood $V\subset U$ of $x$ and for any point $v\in V$, different from $x$, the open set $V-\{v\}$ is path-connected.
\end{thm}
\begin{proof} We are going to show that for any point $y\in V-\{v\}$, we can connect $y$ to $x$ by a path inside $V-\{v\}$. If $y=x$ this statement is obvious. Hence, we assume that $y\ne x$.
\\ \\
Let $c_1,...,c_{\kappa}$ be all the critical points of $f$, such that the signature of $x$ with respect to them is nontrivial. For any critical point $c$, by Corollary \ref{cor.finite}, there exists some neighborhood $U$ of $x$ in $S_{\infty}$, such that the equivalence class $ind(U,c)$ is equal to the maximal element of $sign(x,c)$. Since $f$ has a finite number of critical points, we can select $U$ is such a way that $ind(U,c)$ is the maximal element of $sign(x,c)$ for every critical point $c$ of $f$. Let $(W(1),W(2),...)$ be a sequence of neighborhoods of $x$ shrinking to $x$ such that $W(1)=V$ and $\bar{W}(i+1)\subset W(i)$ for all $i=1,2,...$. Let $h$ be the maximal positive integer, for which $\bar{W}(h)$ contains $y$. This number exists, because $y\ne x$ and $S_{\infty}$ is a Hausdorff space. For all $i=1,2,...$, let $n(i)$ be the minimal positive integer such that for all $n\ge n(i)$, $(U-W(i))_n$ does not contain any critical points of $f$. All these minimal integers $n(i)$ must exist, because $ind(U,c)=sign(x,c)$ for each critical point $c$, which implies that $ind(U,c)=ind(W(i),c)$, and $f$ has a finite number of critical points. Let $w(i)_{n(i)}$ be a point on the boundary of $\bar{W}(i)_{n(i)}$.
\\ \\
We connect $y_{n(h+1)}$ to $w(h+1)_{n(h+1)}$ by any path $p^h_{n(h+1)}:[0,\frac{1}{2}]\rightarrow (W(h)-W(h+1))_{n(h+1)}$. If $v$ is contained in $W(h)-W(h+1)$, we select this path $p^h_{n(h+1)}$ in such a way, that it avoids the point $v_{n(h+1)}$. Since $(W(h)-W(h+1))_n$, for all $n\ge n(h+1)$,  does not contain any critical points of $f$, the path $p^h_{n(h+1)}:[0,\frac{1}{2}]\rightarrow (W(h)-W(h+1))_{n(h+1)}$ has a unique lift to the path $p^h_{n}:[0,\frac{1}{2}]\rightarrow (W(h)-W(h+1))_{n}$ which connects the point $y_n$ to some point $w(h+1)_n$ in the boundary of $\bar{W}(h+1)_{n}$. This point $w(h+1)_n$ must belong to the boundary of $\bar{W}(h+1)_{n}$ because it follows from the Riemann–Hurwitz formula, similarly to how it was applied in the proof of Lemma \ref{lemma.neighborhood.nomax.3}, that $f^{n(h+1)-n}(W(h+1)_{n(h+1)})$ has only one connected component inside $V_n$, which is $W(h+1)_{n}$. Now, for all positive integers $n<n(h+1)$ we define $p^h_{n}=(f^{n(h+1)-n}\circ p^h_{n(h+1)}):[0,\frac{1}{2}]\rightarrow (W(h)-W(h+1))_{n}$. Thus, we obtain a path $p^h:[0,\frac{1}{2}]\rightarrow (W(h)-W(h+1))$, which connects $y$ to a point $w(h+1)$ in the boundary of $\bar{W}(h+1)$.
\\ \\
Next, we connect $w(h+1)_{n(h+2)}$ to $w(h+2)_{n(h+2)}$ by any path $p^{h+1}_{n(h+2)}:[\frac{1}{2},\frac{3}{4}]\rightarrow (\bar{W}(h+1)-W(h+2))_{n(h+2)}$. If $v$ is contained in $\bar{W}(h+1)-W(h+2)$, we select this path $p^{h+1}_{n(h+2)}$ in such a way, that it avoids the point $v_{n(h+2)}$. Since $(\bar{W}(h+1)-W(h+2))_n$, for all $n\ge n(h+2)$, does not contain any critical points of $f$, the path $p^{h+1}_{n(h+2)}:[\frac{1}{2},\frac{3}{4}]\rightarrow (\bar{W}(h+1)-W(h+2))_{n(h+2)}$ has a unique lift $p^{h+1}_{n}:[\frac{1}{2},\frac{3}{4}]\rightarrow (\bar{W}(h+1)-W(h+2))_{n}$ which connects the point $w(h+1)_n$ to some point $w(h+2)_n$ in the boundary of $\bar{W}(h+2)_{n}$. Now, for all positive integers $n<n(h+2)$ we define $p^{h+1}_{n}=(f^{n(h+2)-n}\circ p^{h+1}_{n(h+2)}):[\frac{1}{2},\frac{3}{4}]\rightarrow (\bar{W}(h+1)-W(h+2))_{n}$. Thus, we obtain a path $p^{h+1}:[\frac{1}{2},\frac{3}{4}]\rightarrow (\bar{W}(h+1)-W(h+2))$, which connects $w(h+1)$ to a point $w(h+2)$ in the boundary of $\bar{W}(h+2)$.
\\ \\
Next, we connect $w(h+2)_{n(h+3)}$ to $w(h+3)_{n(h+3)}$ by any path $p^{h+2}_{n(h+3)}:[\frac{3}{4},\frac{7}{8}]\rightarrow (\bar{W}(h+2)-W(h+3))_{n(h+3)}$. If $v$ is contained in $\bar{W}(h+2)-W(h+3)$, we select this path $p^{h+1}_{n(h+3)}$ in such a way, that it avoids the point $v_{n(h+3)}$. Since $(\bar{W}(h+2)-W(h+3))_n$, for all $n\ge n(h+3)$, does not contain any critical points of $f$, the path $p^{h+2}_{n(h+3)}:[\frac{3}{4},\frac{7}{8}]\rightarrow (\bar{W}(h+2)-W(h+3))_{n(h+3)}$ has a unique lift $p^{h+2}_{n}:[\frac{3}{4},\frac{7}{8}]\rightarrow (\bar{W}(h+2)-W(h+3))_{n}$ which connects the point $w(h+2)_n$ to some point $w(h+3)_n$ in the boundary of $\bar{W}(h+3)_{n}$. For all positive integers $n<n(h+3)$ we define $p^{h+2}_{n}=(f^{n(h+3)-n}\circ p^{h+2}_{n(h+3)}):[\frac{3}{4},\frac{7}{8}]\rightarrow (\bar{W}(h+2)-W(h+3))_{n}$. Thus, we obtain a path $p^{h+2}:[\frac{3}{4},\frac{7}{8}]\rightarrow (\bar{W}(h+2)-W(h+3))$, which connects $w(h+2)$ to a point $w(h+3)$ in the boundary of $\bar{W}(h+3)$.
\\ \\
Continuing this way, we, for all $b=1,2,...$, construct paths $p^{h+b}:[1-(\frac{1}{2})^b,1-(\frac{1}{2})^{b+1}]\rightarrow (\bar{W}(h+b)-W(h+b+1))$, which connect the point $w(h+b)$ in the boundary of $\bar{W}(h+b)$ to the point $w(h+b+1)$ in the boundary of $\bar{W}(h+b+1)$. By our construction, all these paths avoid the point $v$. Finally, we define the path $p:[0,1]\rightarrow V-\{v\}$ by:
\begin{itemize}
\item For $0\le t < \frac{1}{2}$ set $p(t)=p^h(t)$;
\item For each $b=1,2,...,$, for $1-(\frac{1}{2})^b\le t< 1-(\frac{1}{2})^{b+1}$ set $p(t)=p^h(t)$;
\item For $t=1$ set $p(1)=x$.
\end{itemize}
It is straightforward to verify, that the path $p$ is contained inside $V$, connects $y$ and $x$, and avoids $v$.
\end{proof}
\section{Signatures $-$ Parabolic and Cremer cycle cases}
It has been established in Theorem 33 of \cite{CCG} that the invariant lifts of attracting and super-attracting cycles to $S_{\infty}$ are irregular points, whose signature with respect to any critical point has a maximal element. Likewise, it was established there, that the invariant lift of a parabolic cycle is an irregular point, which, with respect to some critical point, has a nontrivial signature with a maximal element.
\\ \\
In this section we show, that for some cases of the invariant lift of the parabolic cycle, its signature with respect to some critical point might have no maximal element. Furthermore, we give a necessary condition for the invariant lift to have only signatures with maximal elements. Additionally, we investigate signatures of points, belonging to the invariant lift of the boundary of immediate basin of attraction of certain parabolic cycles. We also show, that the invariant lift of the Cremer cycle always has a signature with no maximal element with respect to some critical point.
\\ \\
Whenever a function $f:S_0\rightarrow S_0$ have a cycle of period $n$ of a certain type, the function $f^n:S_0\rightarrow S_0$ has $n$ fixed points of the same type. Thus, we will investigate here the cases of parabolic and Cremer fixed points, but our results apply in the general situation of cycles of period $n$.
\\ \\
Let $f:S_0\rightarrow S_0$ have a parabolic fixed point $x_0\in S_0$ and let $x=(x_1=x_0,x_2=x_0,...)$ be the invariant lift of $x_0$ to $S_{\infty}$. Assume, that the boundary $\delta B$ of the immediate basin of attraction $B$ of $x_0$ does not contain any critical points of $f$. Then it follows from the Leau-Fatou flower theorem, that there exists a ``small-enough" open neighborhood $U_1$ of $x_1$, such that for any open neighborhood $V_\subset U_1$ of $x_1$ and any point $y_0\in S_0=S_1=S_2=...$:
\begin{itemize}
\item If $y_0\notin \bar{B}$, then there exists an integer $n$ such that $y_0\notin V_i$ for all $i\ge n$;
\item If $y_0\in B$, then there exists an integer $n$ such that $y_0\in V_i$ for all $i\ge n$.
\end{itemize}
Here $V_i$ is the pullback of $V_1$ along $x$. Thus, we obtain the following Lemma:
\begin{lem} \label{parab.max} If the boundary $\delta B$ of the immediate basin of attraction $B$ of a parabolic fixed point $x_0$ does not contain a critical point then all the signatures of the invariant lift $x$ of $x_0$ have maximal elements.
\end{lem}
\begin{proof}
For the critical points $c$, contained in $B$, we will have $ind(V,c)=[1,1,1,...]$ for all ``small enough" neighborhoods $V$ of $x$. Thus, $sign(x,c)=[1,1,1,...]$. For the critical points $c$, not contained in $\bar{B}$, we will have $ind(V,c)=[0,0,0,...]$ for all ``small enough" neighborhoods $V$ of $x$. Thus, $sign(x,c)=[0,0,0,...]$.
\end{proof}
\begin{lem} \label{parab.max.2}  If the boundary $\delta B$ does contain some critical points, but they are all pre-periodic, then all the signatures of the invariant lift $x$ of $x_0$ have maximal elements.
\end{lem}
\begin{proof} Let $c\in \delta B$ be pre-periodic. Thus, for some integer $m>0$, $f^m(c)$ belongs to some cycle, while $f^{m-1}(c)$ does not. If $f^m(c)\ne x_0$ then we can take a small enough neighborhood of $x_0$, so that none of the points of the cycle which contains $f^m(c)$ are contained in that neighborhood. Clearly, all the pre-images of that neighborhood along $x$ will not contain $c$. Suppose that $f^m(c)=x_0$. For any neighborhood $U$ of $x$ all the levels of $U$, starting from some $U_n$ and up, will contain all the inner critical points of $B$. Hence, $U_n$ will contain all the pre-images of $x_0$ in $\delta B$. And $U_{n+1}$ will contain all the pre-images of all these pre-images in $\delta B$. And so on. Thus, $c$ will be contained in all the sets $U_{n+m},U_{n+m+1},U_{n+m+2},...$, which implies that $sign(x,c)=[1,1,1,...]$.
\end{proof}
Assume now that a critical point $c$, contained in the boundary of the immediate basin of attraction of a parabolic cycle, is not pre-periodic. From Theorem 1 of \cite{RoeschYin} it follows that for a polynomial function $f(z)$ the boundary $\delta B$ of a connected component $B$ of the immediate basin of attraction of a parabolic cycle is a Jordan curve. Passing from $f(z)$ to some finite iterate of $f(z)$, permits us to assume that $f(B)=B$. Thus, the Riemann conformal homeomorphism $\phi:B\rightarrow \mathbb{D}$ from $B$ onto the open unit disk, which conjugates $f:B\rightarrow B$ with some finite Blaschke product $\beta:\mathbb{D}\rightarrow\mathbb{D}$, can be extended to a continuous homeomorphism $\phi:\bar{B}\rightarrow \bar{\mathbb{D}}$, which also conjugates $f$ and $\beta$. Furthermore, $\phi$ can be selected in such a way, that it takes the parabolic point to $+1$. Then $+1\in \mathbb{D}$ will be the parabolic Denjoy-Wolff point of $\beta$ (see page 2 in \cite{Poggi}). 
\begin{lem} \label{Blaschke.boundary} Let $\beta:\mathbb{D}\rightarrow\mathbb{D}$ be a Blaschke product of a finite degree $d$ with a Denjoy-Wolff point at $+1$. Then there exists a continuous automorphism $\varphi:S^{1}\rightarrow:S^{1}$, which conjugates $\beta$ and $z\mapsto z^d$ on the boundary of $\mathbb{D}$, and $\varphi(+1)=+1$.
\end{lem}
\begin{proof} Recall, that a finite Blaschke product of degree $d$ is of the form 
$$\beta(z)=\frac{a_dz^d+...+a_0}{\bar{a_0}z^d+...+\bar{a_d}}=e^{2\pi\vartheta\cdot i}\cdot \prod\limits_{t=1}^d \frac{z-b_t}{1-\bar{b_t}z},$$
where all the solutions $b_1,...,b_d$ of the equation $\beta(z)=0$ are contained inside $\mathbb{D}$. By Gauss-Lucas theorem (see \cite{Marden} - Theorem 6.1), all the critical points of $\beta:\bar{\mathbb{D}}\rightarrow \bar{\mathbb{D}}$ are contained inside the convex hull of its zeros $b_1,...,b_d$, which is contained inside $\mathbb{D}$. Thus, $\beta:S^{1}\rightarrow:S^{1}$ is a $d$-fold covering map with no branching points. Moreover, it is easy to see that, since a point $z$ in the Riemann sphere has $d$ pre-images under $\beta^{-1}$ counting them with multiplicities, $\beta^{-1}(S^{1})=S^{1}$. Since $S^{1}$ belongs to the Julia set of $\beta$, this implies, that for any open angular arc $(e^{2\pi(\theta)\cdot i},e^{2\pi(\theta+\epsilon)\cdot i})\subset S^{1}$, the union of its images under the forward iterations of $\beta$ will cover the entire $S^{1}$ except, maybe, one or two points.
\\ \\
Set $\varphi^{-1}(e^{2\pi(\frac{j}{d})\cdot i})$, for $j=0,...,d-1$, to be the $d$ solutions of $\beta(z)=+1$, taken, starting with $+1$, counterclockwise. Clearly, for each $j=1,...,d-1$, the equation $\beta(z)=\varphi^{-1}(e^{2\pi(\frac{j}{d})\cdot i})$ has $d$ different solutions, each belonging to the different open arc $(\varphi^{-1}(e^{2\pi(\frac{t}{d})\cdot i}),\varphi^{-1}(e^{2\pi(\frac{t+1}{d})\cdot i})$, where $t=0,...,d-1$. Indeed, each open arc is mapped by $\beta$ onto $S^{1}-\{+1\}$. Set $\varphi^{-1}(e^{2\pi(\frac{j_1d+j_2}{d^2})\cdot i})$, for $j_1=0,...,d-1$ and $j_2=1,...,d_1$, to be the solution of $\beta(z)=\varphi^{-1}(e^{2\pi(\frac{j_2}{d})\cdot i})$ contained in the arc $(\varphi^{-1}(e^{2\pi(\frac{j_1}{d})\cdot i}),\varphi^{-1}(e^{2\pi(\frac{j_1+1}{d})\cdot i})$. It is easy to check, that $\beta^2$ takes each open arc $(\varphi^{-1}(e^{2\pi(\frac{j_1d+j_2}{d^2}\cdot i}),\varphi^{-1}(e^{2\pi(\frac{j_1d+j_2+1}{d^2})\cdot i})$, for $j_1,j_2=0,...,d-1$, onto $S^{1}-\{+1\}$. We proceed this way to define $\varphi^{-1}$ for all $e^{2\pi t\cdot i}$ with the the rational numbers $t$ of the form $\frac{j_1d^{n-1}+j_2d^{n-2}+...+j_n}{d^n}$ for all $n=1,2,...$, where $j_1,...,j_n=0,...,d-1$. The continuity of this map follows from the fact, that both $\varphi^{-1}(e^{2\pi(t+\frac{1}{d^n})\cdot i})$ and $\varphi^{-1}(e^{2\pi(t-\frac{1}{d^n})\cdot i})=\varphi^{-1}(e^{2\pi(t-1+\frac{(d-1)d^{n-1}+(d-1)^d{n-1}+...+(d-1)}{d^n})\cdot i})$, for any rational number $t$ of the above described form, are converging to $\varphi^{-1}(e^{2\pi t\cdot i})$ as $n\rightarrow\infty$. Indeed, if $\varphi^{-1}(e^{2\pi(t+\frac{1}{d^n})\cdot i})$ converges to some $e^{2\pi(t+\epsilon)\cdot i}$ we would obtain an arc $(e^{2\pi(t)\cdot i},e^{2\pi(t+\frac{1}{2}\epsilon)\cdot i})\subset S^{1}$ such that the union of all its forward iterates does not contain some open arc in $S^{1}$.
\\ \\
Next, we extend $\varphi^{-1}$, using continuity, to all $e^{2\pi(t)\cdot i}$ for any real $0\le t\le 1$. By its construction, $\varphi^{-1}$ conjugates $z\mapsto z^d$ and $\beta$ at all $d^n$ roots of $+1$ for all $n=0,1,2,...$. Thus, $\varphi$ conjugates $\beta$ with $z\mapsto z^d$ for all the points of $S^1$.
\end{proof}
Thus, to compute the signature of the invariant lift $x$ of the parabolic fixed point $x_0$, we need to select any sequence of decreasing positive numbers $(\epsilon_1,\epsilon_2,...)$, which converges to $0$, and define $b_t\in I$, for $t=1,2,...$, as the class the binary sequence which contains $1$ in its $q^{th}$ place if and only if $|\varphi(\phi(c))^{d^q}-1|<\epsilon_t$. If $\varphi(\phi(c))=e^{2\pi(\theta)\cdot i}$ for some irrational $\theta$, then $q$ are all integers, for which the inequality $|d^q\cdot\theta-p|<\epsilon_t$, for some integer $p$, has solutions. Now one applies the arguments, similar to the ones, used in Lemma 37 of \cite{CCG}, to show that $sign(x,c)$ has no maximal element.
\\ \\
Currently, we do not know any examples of critical points on the boundary of the immediate basin of attraction of a parabolic cycle. However, we do not know any argument, why such a point cannot exist or, if it does exist, must be pre-periodic in a polynomial case. Note, that such a non-pre-periodic critical point will be recurrent, since the forward orbits of $z=\varphi(\phi(c))\in S^1$ under $z\mapsto z^d$ will contain $z$ in their $\omega$-limit.
\\ \\
Now we address a generic irregular point $x=(x_1,x_2,...)\in S_{\infty}$, which is not an invariant lift of a super-attracting, an attracting or a parabolic cycle. Recall Lemma 3.5 from \cite{LM}, which can be viewed as a stronger version of Ma$\tilde{\textrm{n}}\acute{\textrm{e}}$'s Theorem II (b) from \cite{Mane}. Here we state and prove a slightly stronger version of these two results. Just like in the proof of Ma$\tilde{\textrm{n}}\acute{\textrm{e}}$, on page 2 of \cite{Mane}, we too can assume that $f(\infty)=\infty$ and that $x_1$ has a neighborhood, which does not contain $\infty$. Thus, we can deal only with subsets of the complex plane. Denote by $disk(z,r)$ an open disk of radius $r$ around a point $z$ in the complex plane.
\begin{thm}\label{Mane.LM} There exists a critical point $c$, such that for any open neighborhood $\Phi=(\Phi_1,\Phi_2,...)\subset S_{\infty}$ of $x$ there exist infinitely many positive numbers $n_1<n_2<...$, so that $\Phi$ contains $c$ in all of its levels $\Phi_{n_i}$, and for every $i=1,2,...$, $c$ is a limit point of the set $\{f^{n_j-n_i}(c)\;\; |\;\; j=i+1,i+2,...\}$.
\end{thm} 
\begin{proof}
Assume that the theorem is wrong. Then exists some $\delta_0>0$ such that for any open neighborhood $\Phi$ there is no critical point $c$, for which exist integers $0<t_1\le t_2$ satisfying
$$c\in f^{-t_1}(disk(c,\delta_0))\cap \Phi_{t_1}$$
$$c\in f^{-t_2}(disk(x_1,\delta_0))\cap \Phi_{t_2}.$$
Comparing this to part 1) of Ma$\tilde{\textrm{n}}\acute{\textrm{e}}$'s proof on page 5 of \cite{Mane}, taking $\delta$ in part 6) of that proof to be small enough, that the closure of the entire square of center $x$ and radius $\delta+2\cdot\frac{\delta}{4}+2\cdot\frac{\delta}{16}+2\cdot\frac{\delta}{64}+...$ is contained inside $\Phi_1$, and redefining $c(U,n)$ on page 2 of \cite{Mane}, for all open $U\subset \Phi_1$, to be defined now as the set of all connected components of $f^{-n}(U)\cap \Phi_n$, permits us to repeat Ma$\tilde{\textrm{n}}\acute{\textrm{e}}$'s argument for our Theorem. The only part, which needs to be justified, is the modified version of Lemma 2 on page 7 of \cite{Mane}. Namely, that if $U\subset \Phi_1$ is an open neighborhood of $x_1$ and $V\in c(U,n)$ satisfies $diam(f^i(V))\le \delta_0$ for $0\le i\le n$ then $\Delta(V,n)\le N_0$.
\\ \\
Before we prove this modified version of Lemma 2, we make two very important notes. The first note is, that, both in Ma$\tilde{\textrm{n}}\acute{\textrm{e}}$'s proof and here, $\Delta(V,n)$ is the number of different points $z\in V$, for which $(f^n)'(z)=0$. The algebraic multiplicity of these points should not be counted. Since the distance between $x_1$ and any super-attracting periodic point of $f$ is greater than some fixed positive number and since $f$ has a finite degree, there exist a neighborhood $\Omega$ of $x_1$ and two fixed positive numbers $\tau_1$ and $\tau_2$ such that for any open set $U\subset \Omega$, any $n$, and any connected component $V$ of $f^{-n}(U)$,  $\tau_1\cdot \Delta(V,n)<\Delta'(V,n)<\tau_2\cdot \Delta(V,n)$, where $\Delta'(V,n)$ is the number of different points $z\in V$, for which $(f^n)'(z)=0$, counted with algebraic multiplicity $-$ the definition, given for $\Delta(V,n)$, on page 2 of Ma$\tilde{\textrm{n}}\acute{\textrm{e}}$'s work. Thus, $\Delta'(V,n)$ and $\Delta(V,n)$ are interchangeable in Lemma 1 and, consequently, in 6) on page 5 of \cite{Mane}. The second note is, that, both in Ma$\tilde{\textrm{n}}\acute{\textrm{e}}$'s work and here, $N_0$ should be defined as $k\cdot deg(f)^{k-1}$, where $k$ is the number of different critical points of $f$. Defining $N_0$ just as the number of different critical points of $f$ is not sufficient, since in the proof of Lemma 2 of \cite{Mane} the case $m_1=m_2$ yields $dist(f^0(c),c)\le \delta_0$, which does not contradict the requirement 1) from page 5 of \cite{Mane}. Obviously, that this change in the definition of the fixed number $N_0$ also causes the values of $N_1$ in 5) and of $\delta$ in 6) on page 5 of Ma$\tilde{\textrm{n}}\acute{\textrm{e}}$'s work to change.
\\ \\
If for two integers $0\le i_1<i_2\le n$ and two points $y'_{n+1},y''_{n+1}\in V$ we have $f^{i_1}(y'_{n+1})=f^{i_2}(y''_{n+1})=c$ for some critical point $c$ of $f$, then $c=f^{i_1}(y'_{n+1})$ is contained inside $$f^{i_1}(V)\subset f^{i_1-i_2}(f^{i_2-i_1}(V))\cap \Phi_{n+1-i_1}\subset f^{i_1-i_2}(disk(c,\delta_0))\cap \Phi_{n+1-i_1},$$ while that same $c=f^{i_2}(y''_{n+1})$ is contained inside $$f^{i_2}(V)\subset f^{i_2-n}(f^{n}(V))\cap \Phi_{n+1-i_2}\subset f^{i_2-n}(disk(x,\delta_0))\cap \Phi_{n+1-i_2},$$ which contradicts our choice of $\delta_0$. Thus, each critical point $c$ of $f$ can appear at most in one of the sets $V,f(V),f^2(V),...,f^n(V)$. Since $f$ has $k$ different critical points, and the branching number at each one of them is $\le deg(f)$, we get that at most $deg(f)^{k-1}$ different points $y_{n+1}\in V$ can have $f^i(y_{n+1})=c$ for some $0\le i\le n$. Since $f$ has $k$ different critical points, we see that at most $k\cdot deg(f)^{k-1}$ different points $y_{n+1}\in V$ can have $f^i(y_{n+1})$ equal to some critical point of $f$ for some $0\le i\le n$. But $(f^n)'(y_{n+1})=0$ if and only if $f^i(y_{n+1})$ is equal to some critical point of $f$ for some $0\le i\le n$. This completes the proof of the modified version of Lemma 2 and of our theorem.
\end{proof}
\begin{thm}\label{nonperiodic.signatures} Let $x=(x_1,x_2,...)\in S_{\infty}$ be an irregular point, which is not an invariant lift of a super-attracting, an attracting or a parabolic cycle. Then exists a recurrent critical point $c$, such that $sign(x,c)$ does not have a maximal element
\end{thm}
\begin{proof} By Theorem \ref{Mane.LM}, there exists a critical point $c$, such that for any open neighborhood $\Phi=(\Phi_1,\Phi_2,...)\subset S_{\infty}$ of $x$ there exist infinitely many positive numbers $n_1<n_2<...$, so that $\Phi$ contains $c$ in all of its levels $\Phi_{n_i}$, and for every $i=1,2,...$, $c$ is a limit point of the set $\{f^{n_j-n_i}(c)\;\; |\;\; j=i+1,i+2,...\}$. This $c$, since it is a limit point of the set $\{f^{n_j-n_1}(c)\;\; |\;\; j=2,3,...\}$, is recurrent. For any neighborhood $\Phi$ of $x$ and infinitely many positive numbers $n_1<n_2<...$, as above, take neighborhood $\Psi\subset \Phi$ of $x$ to be such, that its closure $\bar{\Psi}_{n_1}$ does not contain $c$. This is always possible, since $x_{n_1}\ne c$ and $S_{n_i}$ is a regular space. Let $(n_{j_1},n_{j_2},n_{j_3},...)$ be a subsequence of $(n_1,n_2,n_3,...)$ such that the sequence $(f^{n_{j_1}-n_1}(c),f^{n_{j_2}-n_1}(c),f^{n_{j_3}-n_1}(c),...)$ converges to $c$. Then almost all, except a finite number of, the levels $\Psi_{n_{j_1}},\Psi_{n_{j_2}},\Psi_{n_{j_3}},...$ will not contain $c$. Thus, $ind(\Psi,c)<ind(\Phi,c)$, which, by Corollary \ref{cor.finite}, implies that $sign(x,c)$ does not have a maximal element.
\end{proof}
Finally, we address the question of existence of isolated irregular points in plaque inverse limits. By Corollary \ref{nomax.main.cor.1}, plaque inverse limit at such points would be locally connected but not locally path-connected.
\\ \\
It is know from Theorems 4.1 and 4.2 of \cite{CK} that the invariant lift $x$ of a Cremer cycle is not an isolated irregular point. Indeed, for any Cremer fixed point $x_0$ of $f$ we have some open neighborhood $\Omega_0$ of $x_0$ in which $f$ is univalent and in which there exists a univalent branch $g$ of $f^{-1}$ with $g(x_0)=x_0$. Hence, by Theorem 4.1 of \cite{CK}, for any open neighborhood $U_0$ of $x_0$, with $\bar{U}_0\subset \Omega_0$, there exists a compact, connected set $H\subset \bar{U}_0$, containing the Cremer point $x_0$ and one or more points from the boundary of $U_0$, which is a full continuum in a Julia set of $f$ and satisfies $f(H)=H$ and $g(H)=H$. Such $H$ is called a hedgehog of $x_0$. Theorem 4.2 of \cite{CK} asserts that the invariant lift of $H$ to $S_{\infty}$ consists only of irregular points. Thus, any open neighborhood $U$ of $x$, with $\bar{U}_0\subset \Omega_0$, will contain an uncountable number of irregular points, different from $x$.
\\ \\
All the points of the invariant lifts of the boundaries of Siegel disks and Herman rings are irregular (see \cite{LM} and \cite{CCG}). Thus, these invariant lifts do not contain isolated irregular points.
\\ \\
Finally, we consider the invariant lift $x$ of the parabolic fixed point $x_0=0$ of a polynomial function $f(z)$ with a non-pre-period, thus recurrent, critical point $c$ on the boundary $\delta B$ of a connected component $B$ the immediate basin of attraction of $x_0=0$. Assume that $f(B)=B$. The forward orbit of $c$ is dense in $\delta B$. Consider any neighborhood $U_1$ of $x_1=x_0$. After some finite number of pull-backs of $U_1$ along $x$, all the further pullbacks will contain all the critical points inside $B$. Thus, we can assume that $U_1$ already contains all these critical points. Consider any point $y_1\in U_1$ and take any sequence $(V_1(1),V_1(2),...)$ of open neighborhood of $y_1$ converging to $y_1$. We can select some pre-images $y_2$ of $y_1$, $y_3$ of $y_2$, ..., $y_{q_1}$ of $y_{q_1}-1$ in $\delta B$ so, that the pullback $V_{q_1}(1)$ of $V_1(1)$ along $y_{q_1}$ contains $c$. Now, we can select some pre-images $y_{q_1+1}$ of $y_{q_1}$, ..., $y_{q_2}$ of $y_{q_2}-1$ in $\delta B$ so, that the pullback $V_{q_2}(2)$ of $V_1(2)$ along $y_{q_2}$ contains $c$. Continuing this way, we construct an irregular point $y=(y_1,...)$, which is contained in the lift of $U_1$ along $x$. Thus, $x$ is not an isolated irregular point.
\\ \\
Currently, we do not know any examples of an isolated irregular point, which has a signature with no maximal element.

 \end{document}